\documentclass{amsart} 
\usepackage{esint}
\usepackage{amsfonts}
\usepackage{amssymb}
\usepackage{xcolor}
\usepackage{url, hyperref}

\newtheorem{theorem}{Theorem}[section]

\newtheorem{conjecture}[theorem]{Conjecture}

\newtheorem{lemma}[theorem]{Lemma}

\newtheorem{question}[theorem]{Question}
\newtheorem{proposition}[theorem]{Proposition}
\newtheorem{remark}[theorem]{Remark}

\numberwithin{equation}{section}  
 
\begin{document}
\title[Scalar curvature integral bound]{Sharp integral bound of scalar curvature on $3$-manifolds}
\author{Ovidiu Munteanu and Jiaping Wang}

\begin{abstract}
It is shown that the integral of the scalar curvature on a geodesic ball of radius $R$ in a three-dimensional complete manifold with nonnegative Ricci curvature is bounded above by $8\pi R$ asymptotically for large $R$ provided that the scalar curvature is bounded between two positive constants.
\end{abstract}

\address{Department of Mathematics, University of Connecticut, Storrs, CT
06268, USA}
\email{ovidiu.munteanu@uconn.edu}
\address{School of Mathematics, University of Minnesota, Minneapolis, MN
55455, USA}
\email{jiaping@math.umn.edu}

\maketitle

\section{Introduction}

The classical Cohn-Vossen theorem states that the inequality $\int_M K\,dA\leq 2\pi \chi (M)$ holds for every complete Riemann 
surface $M$ with finite total curvature, where $K$ and $\chi (M)$ are the Gaussian curvature and the Euler characteristic
number of $M,$ respectively. As a higher dimensional analogue, Yau \cite{Y} has raised the following question.

\begin{question}\label{QY}
For an $n$-dimensional complete manifold $\left( M^{n},g\right) $ with its Ricci curvature $\mathrm{Ric}\geq 0,$ does its scalar
curvature $S$ satisfy

\begin{equation*}
\limsup_{R\rightarrow \infty }\frac{1}{R^{n-2}}\int_{B_{p}(R)}S<\infty?
\end{equation*}
\end{question}
Here and in the following, $B_p(R)$ denotes the geodesic ball of radius $R$ centered at a point $p\in M.$
Note that if $M$ has more than one end, then the Cheeger-Gromoll splitting theorem implies that $M=\mathbb{R}\times N,$ 
where $N$ is compact. In this case, the answer to the question is clearly affirmative. So we may restrict our attention to manifolds with one end. Our main purpose here is to approach the question in the three dimensional case by establishing the following result. 

\begin{theorem}
\label{A}Let $\left( M^{3},g\right) $ be a three-dimensional complete manifold with $\mathrm{Ric}\geq 0.$
If its scalar curvature $S$ is bounded between two positive constants, then
\begin{equation*}
\limsup_{R\rightarrow \infty }\frac{1}{R}\int_{B_{p}(R)}S\leq 8\pi.
\end{equation*}
\end{theorem}

In view of the result, it seems reasonable to conjecture that the same conclusion holds without any restriction on the scalar curvature.  
Indeed, this has been shown to be true by Xu \cite{X1} if $M$ admits a pole. 

\begin{conjecture}
\label{CY}Let $\left( M^{3},g\right) $ be a three-dimensional complete manifold with Ricci curvature  $\mathrm{Ric}\geq 0.$
Then its scalar curvature $S$ satisfies 
\begin{equation*}
\limsup_{R\rightarrow \infty }\frac{1}{R}\int_{B_{p}(R)}S\leq 8\pi.
\end{equation*}
\end{conjecture}

Note that the quantity on the left hand side is scaling invariant and the constant $8\pi$ on the right hand side is the best possible.
We emphasize that our proof of Theorem \ref{A} actually yields an upper bound of $\int_{B_{p}\left( R\right) }S$ under the localized assumption that 
$S\geq 1$ on $B_{p}\left( 2R\right)$.   For a large enough universal constant $r_0$, we denote with $N_{r_0}$ the  unbounded connected component of $M\setminus \overline{B_p(r_0)}$.  Evidently, the set $D_{r_0}:=M\setminus N_{r_0}$ is compact in $M$ and is equal to the union of $\overline{B_p(r_0)}$ with all bounded connected components of $M\setminus \overline{B_p(r_0)}$. 

\begin{theorem}
\label{B}Let $\left( M^{3},g\right) $ be a complete noncompact
three-dimensional manifold with $\mathrm{Ric}\geq 0.$ Assume that the scalar
curvature of $\left( M,g\right) $ satisfies 
\begin{equation*}
1\leq S\leq K\text{ \ on }B_{p}\left( 2R\right)
\end{equation*}%
for some constant $K>1$ and $R>2.$ Then there exists a universal constant $C>0$ such that
\begin{equation}
\int_{B_{p}\left( R\right) }S\leq 8\pi R+CK\int_{1}^{R}\mathrm{V}\left(
B_{p}\left( s\right) \right) \frac{\ln s}{s^{2}}ds+CK\mathrm{V}(D_{r_0}).  \label{Est}
\end{equation}
\end{theorem}

Theorem \ref{A} follows from Theorem \ref{B} as the volume $\mathrm{V}\left(B_{p}\left( s\right) \right)$ of the ball $B_p(s)$
satisfies $\mathrm{V}\left(B_{p}\left( s\right) \right)\leq Cs$ according to \cite{MW1}.

We now indicate some of the ideas involved in the proof of Theorem \ref{B}. Essential to our argument is the following lemma.
It follows directly from the Gauss curvature equation together with some algebraic manipulations (see e.g. \cite{MW2} for details).
  
\begin{lemma}
\label{RicS}Let $u$ be a smooth function on three-dimensional complete
Riemannian manifold $\left( M, g\right).$ Then on each regular level set $\{u=t\}$ of $u,$

\begin{eqnarray*}
\mathrm{Ric}\left( \nabla u,\nabla u\right) \left\vert \nabla u\right\vert
^{-2} &=&\frac{1}{2}S-\frac{1}{2}S_{t}+\frac{1}{2}\frac{1}{\left\vert \nabla
u\right\vert ^{2}}\left( \left\vert \nabla \left\vert \nabla u\right\vert
\right\vert ^{2}-\left\vert \nabla ^{2}u\right\vert ^{2}\right) \\
&&+\frac{1}{2}\frac{1}{\left\vert \nabla u\right\vert ^{2}}\left( \left(
\Delta u\right) ^{2}-2\frac{\left\langle \nabla \left\vert \nabla
u\right\vert ,\nabla u\right\rangle }{\left\vert \nabla u\right\vert }\Delta
u+\left\vert \nabla \left\vert \nabla u\right\vert \right\vert ^{2}\right) ,
\end{eqnarray*}%
where $S_{t}$ denotes the scalar curvature of the surface $\{u=t\}.$
\end{lemma}

The lemma has its origin in Schoen and Yau \cite{SY}, where they made the important observation that on
a minimal surface $N$ in a three-dimensional manifold $M,$

\begin{equation}
\mathrm{Ric}\left( \nu,\nu\right)=\frac{1}{2}S-\frac{1}{2}S_{N }-\frac{1}{2}%
\left\vert A\right\vert^2  \label{a1}
\end{equation}
with $\nu,$ $S_{N}$ and $A$ being the unit normal vector, the scalar
curvature and the second fundamental form of $N,$ respectively. This
observation played an important role in their study of stable minimal surfaces in a
three-dimensional manifold with nonnegative scalar curvature. 
Later, an identity of the nature (\ref{a1}) was derived for any
surface $N,$ not necessarily minimal, by Jezierski and Kijowski in \cite{JK}, which
obviously includes the lemma as a special case with $N$ being the level set of function $u.$ 
Note that Stern \cite{St} has also established the lemma when $u$ is a harmonic function. 

Now start with the standard Bochner formula
\begin{equation*}
\Delta \left\vert \nabla u\right\vert =\left( \left\vert u_{ij}\right\vert
^{2}-\left\vert \nabla \left\vert \nabla u\right\vert \right\vert
^{2}+\mathrm{Ric}\left( \nabla u,\nabla u\right) \right)\left\vert \nabla u\right\vert ^{-1}+\langle \nabla u, \nabla \Delta u\rangle \left\vert \nabla u\right\vert ^{-1},
\end{equation*}
where $u$ is a smooth proper function. After integrating the equation over the compact domain $\{1<u<R\}$
and invoking the co-area formula, one gets a term of the form 
$\mathrm{Ric}\left( \nabla u,\nabla u\right) \left\vert \nabla u\right\vert^{-2}.$ Rewriting it by the lemma and applying the Gauss-Bonnet formula to the integral of $S_t$ on the level set, one obtains 
an upper estimate of the integral $\int_{\{1<u<R\}} S\left\vert \nabla u\right\vert$ by the Euler characteristic 
numbers of regular level sets $\{u=t\}$ for $1<t<R.$ This argument shows that Question \ref{QY} can be affirmatively answered in dimension three provided that $M$ admits a smooth proper function $u$ with $\left\vert \nabla u\right\vert$ bounded from below by a positive constant and its level sets are connected. This is certainly the case when $M$ admits a pole $p,$ where $u$ is the distance function $r(x)=d(x,p).$ This is exactly how Xu \cite{X1} verifies Conjecture \ref{CY} for this case (see also the work of Zhu \cite{Z}). In general, however, the distance function may not be smooth, and the geodesic spheres may have many connected components. Instead, we construct a smooth function $\rho$ as a replacement. Such function $\rho$ is distance-like in the sense that it is proper and the length of its gradient is almost $1$ except for a small set. This latter fact is sufficient to convert the estimate of $\int_{\{1<\rho<R\}} S\,\left\vert \nabla \rho\right\vert$ into one  for $\int_{\{1<\rho<R\}} S.$ However, we still face one serious obstacle for successfully carrying out the argument, namely, the control of the number of connected components of the level sets of $\rho.$ This is where the assumption on $S$ comes into play. Indeed, we divide each level set $\{\rho=t\}$ into two disjoint parts $\ell_0(t)$ 
and $\ell_1(t),$ with $\ell_0(t)$ being the boundary of the only unbounded component of $M\setminus \{\rho<t\}$ and $\ell_1(t)$ the rest. It can be shown that $\ell_0(t)$ is always connected. Also, $\ell_1(t)$ consists of the boundary of all bounded components of $M\setminus \{\rho<t\}$, which we call fingers. The upper bound of $S$ is then used to estimate the integral of the Ricci curvature term on $\ell_1(t),$ while Lemma \ref{RicS} and the Gauss-Bonnet formula are applied to control it on $\ell_0(t).$  The proof of Theorem \ref{B} is then completed by
showing that the size of the fingers are all small relative to $R,$ for which the lower bound of $S$ is required in order to 
appeal to the $\mu$-bubble diameter estimate in \cite{CLS}.

It should be noted that in some previous work we explored the idea of combining the Bochner formula with Lemma \ref{RicS} with $u$ primarily being a harmonic function. In \cite{MW3} a sharp monotonicity formula of the form
\begin{equation*}
\frac{d}{dt}\left( \frac{1}{t}\int_{\ell\left( t\right) }\left\vert \nabla G\right\vert ^{2}-4\pi t\right) \leq 0
\end{equation*}
is derived for a proper, positive Green's function $G$ with $M$ being simply connected of one end and scalar curvature 
$S\geq 0$; in \cite{MW2} a sharp upper bound for the bottom spectrum $\lambda_1(M)$ of the form  $\lambda_1(M)\leq 1$ is established
for $M$ with finitely many ends, finite first Betti number and $S\geq -6$. In \cite{MW1} a volume estimate of the form 
$\mathrm{V}(B_p(R))\leq C\,R$ for all $R\geq 1$ is proved for $M$ with $\mathrm{Ric}\geq 0$ and $S\geq 1$; see \cite{CLS} for an alternative proof, \cite{Wa} for a localized version and \cite{X, HL} for a sharp version.  

One obvious advantage of a harmonic function is that its level sets within each end are necessarily connected due to the maximum principle. However, the drawback is that it is unclear how to obtain a workable lower bound
of the length of its gradient for the purpose of tackling Question \ref{QY}. This forces us to work instead with the distance-like function $\rho.$ 

Lastly, we mention that substantial progress has been made by Jiang and Naber (see \cite{N}), Smith \cite{S} and Zhu \cite{Z}  
toward Question \ref{QY} in arbitrary dimension. 

\begin{remark}
   Theorem \ref{B} can be generalized to bounds of $S$ that depend on distance function, such as  $\frac{1}{C_0 r(x)^a}\leq S(x)\leq C_0 r(x)^b$, where $a,b>0$. It can be seen that for small enough universal constants $a,b>0$, Theorem \ref{A} remains valid. 
\end{remark}

The paper is organized as follows. In section \ref{2}, we provide details on the construction of the distance like function $\rho.$
This construction is not dimension specific and may be useful for other purposes. In section \ref{3}, the primary focus is on 
the level sets of  $\rho.$ Mainly, it is shown that the size of fingers is well-controlled in terms of the scalar curvature lower bound.
Section \ref{4} is devoted to the proof of the main result Theorem \ref{B}. \\

\textbf{Acknowledgment:}   The first author was
partially supported by a Simons Foundation
grant.

\section{Construction of a distance-like function}\label{2}
 
In this section, we construct a smooth distance-like function. Such construction originated in \cite{SY94} and was further
developed in \cite{BS}. While the construction is not dimension specific, we choose to write down the details in the 
three-dimensional case. In the following, we use $C$ to denote a universal constant whose value may change
from line to line, $\mathrm{V}\left( \Omega \right) $ the volume of a set $\Omega\subset M,$ 
$\mathrm{A}\left( \Sigma \right) $ the area of a surface $\Sigma\subset M,$ and $r(x):=d(p,x)$ the distance to 
a fixed point $p\in M.$  Also,  we let $R_{0}>1$ be a large enough universal constant.

\begin{proposition}
\label{rho}Let $\left( M^{3},g\right) $ be a complete noncompact three-dimensional manifold with $\mathrm{Ric}\geq 0.$ 
Then there exists a
smooth function $\rho >0$ on $M\setminus B_{p}\left(1\right) $ such that
\begin{equation*}
\Delta \rho =-1+\left\vert \nabla \rho \right\vert ^{2}.
\end{equation*}
Moreover,
\begin{equation*}
r-3\ln r\leq \rho \leq r+6\ln r
\end{equation*}
and 
\begin{equation*}
\left\vert \nabla \rho \right\vert ^{2}-1\leq \frac{10}{r}\left\vert \nabla
\rho \right\vert ^{2}
\end{equation*}
on $M\setminus B_{p}\left(R_0\right)$ for some large constant $R_0$.
\end{proposition}

\begin{proof}
Following the construction in Ch.1 of  \cite{SY94}, we solve the boundary value problem%
\begin{eqnarray*}
\Delta h_{R} &=&h_{R}\text{ \ in }B_{p}\left( R\right) \setminus B_{p}\left(
1\right) \\
h_{R} &=&1\text{ \ on }\partial B_{p}\left( 1\right) \\
h_{R} &=&0\text{ \ on }\partial B_{p}\left( R\right) .
\end{eqnarray*}%
The maximum principle implies that $0<h_{R}\leq 1$ and that $h_{R}$ is
increasing in $R$. Letting $R\rightarrow \infty ,$ we obtain a positive
function $0<h\leq 1$ on $M\setminus B_{p}\left( 1\right) $ satisfying%
\begin{eqnarray}
\Delta h &=&h\text{ \ in }M\setminus B_{p}\left( 1\right)  \label{d0} \\
h &=&1\text{ \ on }\partial B_{p}\left( 1\right) .  \notag
\end{eqnarray}%
The Cheng-Yau gradient estimate \cite{CY}, see also Ch.6 in \cite{Li}, implies that 
\begin{equation*}
\left\vert \nabla \ln h\right\vert ^{2}\leq 1+\frac{C}{\sqrt{r}}\text{ \ on }%
M\setminus B_{p}\left( 2\right) .
\end{equation*}%
The function 
\begin{equation*}
\rho \left( x\right) =-\ln h\left( x\right)
\end{equation*}%
therefore satisfies the equation%
\begin{equation}
\Delta \rho =-1+\left\vert \nabla \rho \right\vert ^{2}  \label{eq_rho}
\end{equation}%
and the estimate 
\begin{equation}
\left\vert \nabla \rho \right\vert ^{2}-1\leq \frac{C}{\sqrt{r}}\text{ \ on }%
M\setminus B_{p}\left( 2\right) .  \label{d1}
\end{equation}%
We first establish a linear lower bound for $\rho $. Indeed, integrating by
parts and using that $h_{R}=0$ on $\partial B_{p}\left( R\right) $, we have

\begin{eqnarray*}
\int_{B_{p}\left( R\right) \setminus B_{p}\left( 2\right) }h_{R}^{2}e^{2r-\ln
r} &=&\int_{B_{p}\left( R\right) \setminus B_{p}\left( 2\right)
}h_{R}(\Delta h_{R})e^{2r-\ln
r} \\
&=&-\int_{B_{p}\left( R\right) \setminus B_{p}\left( 2\right) } \left\vert \nabla h_{R}\right\vert ^{2} e^{2r-\ln
r}\\
&&-2\int_{B_{p}\left( R\right) \setminus B_{p}\left( 2\right) }\left\langle
\nabla h_{R},\nabla r\right\rangle \left( 1-\frac{1}{2r}\right)
h_{R}e^{2r-\ln r} \\
&&-e^{4-\ln 2}\int_{\partial B_{p}\left( 2\right) }h_{R}\frac{\partial h_{R}%
}{\partial \nu }.
\end{eqnarray*}%
Note that

\begin{eqnarray*}
&&-2\int_{B_{p}\left( R\right) \setminus B_{p}\left( 2\right) }\left\langle
\nabla h_{R},\nabla r\right\rangle \left( 1-\frac{1}{2r}\right)
h_{R}e^{2r-\ln r} \\
&\leq& \int_{B_{p}\left( R\right) \setminus B_{p}\left( 2\right) } \left\vert \nabla h_{R}\right\vert ^{2}e^{2r-\ln r}
+\int_{B_{p}\left( R\right) \setminus B_{p}\left( 2\right) }\left( 1-\frac{%
1}{2r}\right) ^{2} h_{R}^{2} e^{2r-\ln r}.
\end{eqnarray*}%
Moreover, (\ref{d1}) and the volume comparison theorem for $\mathrm{Ric}\geq 0$ imply that $$\left\vert \int_{\partial B_{p}\left(
2\right) }h_{R}\frac{\partial h_{R}}{\partial \nu }\right\vert \leq C\mathrm{A}(\partial B_p(2))\leq C\mathrm{V}(B_p(1)).$$
Therefore, 

\begin{equation*}
\int_{B_{p}\left( R\right) \setminus B_{p}\left( 2\right) } h_{R}^{2}e^{2r-\ln
r}\leq \int_{B_{p}\left( R\right) \setminus B_{p}\left( 2\right)
}\left( 1-\frac{1}{2r}\right) ^{2}h_{R}^{2}e^{2r-\ln r}+C\mathrm{V}(B_p(1)).
\end{equation*}
Simplifying the above inequality we get%
\begin{equation*}
\int_{B_{p}\left( R\right) \setminus B_{p}\left( 2\right) } h_{R}^{2}e^{2r-2\ln
r}\leq C\mathrm{V}(B_p(1))
\end{equation*}%
for all $R>2$. Taking $R\rightarrow \infty $, we arrive at

\begin{equation*}
\int_{M\setminus B_{p}\left( 2\right) } h^{2}e^{2r-2\ln r}\leq
C\mathrm{V}(B_p(1)).
\end{equation*}%
Together with (\ref{d1}) we get that%
\begin{eqnarray*}
h^{2}\left( x\right) &\leq &\frac{C}{\mathrm{V}\left( B_{x}\left( 1\right)
\right) }\int_{B_{x}\left( 1\right) }h^{2} \\
&\leq &\frac{C\mathrm{V}(B_p(1))}{\mathrm{V}\left( B_{x}\left( 1\right) \right) }e^{-2r\left(
x\right) +2\ln r\left( x\right) }
\end{eqnarray*}%
for every $x\in M\setminus B_{p}\left( 3\right) $. Since $\mathrm{Ric}\geq 0,$ 
by the Bishop-Gromov volume comparison theorem,  
$$
\frac{\mathrm{V}(B_p(1))}{\mathrm{V}\left( B_{x}\left( 1\right) \right) }\leq Cr(x)^3.
$$
Hence, we conclude that $h$ decays exponentially and
\begin{equation*}
h\left( x\right) \leq e^{-r\left( x\right) +3\ln r\left( x\right) }\text{ \
on }M\setminus B_{p}\left( r_{0}\right) 
\end{equation*}
for $r_0>C$  large enough. 
Since $\rho =-\ln h$, this proves that 
\begin{equation}
\rho \geq r-3\ln r\text{ \ on }M\setminus B_{p}\left( r_{0}\right) .
\label{lb}
\end{equation}%

To establish an analogous upper bound for $\rho $ we first improve (\ref{d1}%
) to the form%
\begin{equation*}
\left\vert \nabla \rho \right\vert ^{2}-1\leq \frac{C}{r}.
\end{equation*}
The Bochner formula and the equation $\Delta \rho =-1+\left\vert \nabla \rho
\right\vert ^{2}$ imply that 
\begin{eqnarray*}
\frac{1}{2}\Delta \left\vert \nabla \rho \right\vert ^{2} &=&\left\vert
\nabla ^{2}\rho \right\vert ^{2}+\mathrm{Ric}\left( \nabla \rho ,\nabla \rho
\right) +\left\langle \nabla \Delta \rho ,\nabla \rho \right\rangle \\
&\geq &\frac{1}{3}\left( \Delta \rho \right) ^{2}+\left\langle \nabla \rho
,\nabla \left\vert \nabla \rho \right\vert ^{2}\right\rangle .
\end{eqnarray*}%
The function $\sigma :=\left\vert \nabla \rho \right\vert ^{2}$ therefore
satisfies 
\begin{equation}
\frac{1}{2}\Delta \sigma \geq \frac{1}{3}\left( \sigma -1\right)
^{2}+\left\langle \nabla \rho ,\nabla \sigma \right\rangle .  \label{d3}
\end{equation}
We may assume by (\ref{d1}) that 
\begin{equation}
\left\vert \nabla \rho \right\vert ^{2}-1\leq \frac{1}{10}\text{ \ on }%
M\setminus B_{p}\left( r_{0}\right)  \label{d4}
\end{equation}%
by taking $r_{0}$ large enough. For constants $a$,$b>0$ to be specified
later, let 
\begin{equation*}
G:=\sigma -\frac{a}{\rho }-\frac{b}{\rho ^{2}}.
\end{equation*}%
We have 
\begin{eqnarray*}
\Delta \left( \frac{a}{\rho }+\frac{b}{\rho ^{2}}\right) &=&-\left( \frac{a}{%
\rho ^{2}}+\frac{2b}{\rho ^{3}}\right) \Delta \rho +\left( \frac{2a}{\rho
^{3}}+\frac{6b}{\rho ^{4}}\right) \left\vert \nabla \rho \right\vert ^{2} \\
&\geq &-\left( \frac{a}{\rho ^{2}}+\frac{2b}{\rho ^{3}}\right) \Delta \rho .
\end{eqnarray*}
As (\ref{d4}) and (\ref{eq_rho}) imply that $\Delta \rho \leq \frac{1}{10},$
it follows that 
\begin{equation*}
\Delta \left( \frac{a}{\rho }+\frac{b}{\rho ^{2}}\right) \geq -\frac{1}{10}%
\left( \frac{a}{\rho ^{2}}+\frac{2b}{\rho ^{3}}\right) .
\end{equation*}
By (\ref{d3}) we see that $G$ satisfies 
\begin{equation*}
\frac{1}{2}\Delta G\geq \frac{1}{3}\left( G-1+\frac{a}{\rho }+\frac{b}{\rho
^{2}}\right) ^{2}-\frac{1}{20}\left( \frac{a}{\rho ^{2}}+\frac{2b}{\rho ^{3}}%
\right) +\left\langle \nabla \rho ,\nabla \sigma \right\rangle .
\end{equation*}%
Note that 
\begin{equation*}
\left\langle \nabla \rho ,\nabla \sigma \right\rangle =\left\langle \nabla
\rho ,\nabla G\right\rangle -\left( \frac{a}{\rho ^{2}}+\frac{2b}{\rho ^{3}}%
\right) \left\vert \nabla \rho \right\vert ^{2}.
\end{equation*}
Using (\ref{d4}), we conclude that 
\begin{equation}
\frac{1}{2}\Delta G\geq \frac{1}{3}\left( G-1+\frac{a}{\rho }+\frac{b}{\rho
^{2}}\right) ^{2}+\left\langle \nabla G,\nabla \rho \right\rangle -\frac{4}{3%
}\left( \frac{a}{\rho ^{2}}+\frac{2b}{\rho ^{3}}\right) .  \label{d5}
\end{equation}

Now set 
\begin{equation*}
a=9\text{ \ and }b=2r_{0}^{2}.
\end{equation*}
Let $x_{0}\in \overline{B_{p}\left( R\right) }\setminus B_{p}\left(
r_{0}\right) $ be a maximum point of $G$ in $\overline{B_{p}\left( R\right) }%
\setminus B_{p}\left( r_{0}\right) .$ If $x_{0}\in \partial B_{p}\left(
r_{0}\right) $, then $G\left( x_{0}\right) <0$ by (\ref{d4}) and the choice of $b$.
If $x_{0}\in \partial B_{p}\left( R\right) $, then $G\left( x_{0}\right)
\leq 1+\frac{C}{\sqrt{R}}$ by (\ref{d1}).
Finally, assume that $x_{0}\in
B_{p}\left( R\right) \setminus \overline{B_{p}\left( r_{0}\right) }$ is an
interior point. Then $\Delta G\left( x_{0}\right) \leq 0$ and $\nabla
G\left( x_{0}\right) =0$. Hence, (\ref{d5}) implies that 
\begin{equation*}
\left( G-1+\frac{a}{\rho }+\frac{b}{\rho ^{2}}\right) ^{2}\leq 4\left( \frac{%
a}{\rho ^{2}}+\frac{2b}{\rho ^{3}}\right) \text{ \ \ at }x_{0}.
\end{equation*}%
Consequently, 
\begin{equation*}
G\leq 1-\frac{a}{\rho }-\frac{b}{\rho ^{2}}+2\left( \frac{\sqrt{a}}{\rho }+%
\frac{\sqrt{2b}}{\rho ^{\frac{3}{2}}}\right) \text{ \ \ at }x_{0}.
\end{equation*}%
Therefore,

\begin{equation*}
G\leq 1-\frac{a-2\sqrt{a}-2}{\rho }\text{ \ at }x_{0}
\end{equation*}
as
\begin{equation*}
\frac{2\sqrt{2b}}{\rho ^{\frac{3}{2}}}\leq \frac{b}{\rho ^{2}}+\frac{2}{\rho}.
\end{equation*}

Since $a=9,$ we see that in this case $G\left( x_{0}\right) \leq 1$. In
conclusion, we have proved that $G\leq 1+\frac{C}{\sqrt{R}}$ on 
$B_{p}\left( R\right) \setminus B_{p}\left(r_{0}\right).$ 
Taking $R\rightarrow \infty$ implies $G\leq 1$ on 
$M\setminus B_{p}\left( r_{0}\right).$ In view of (\ref{lb}), this immediately leads to

\begin{equation}
\left\vert \nabla \rho \right\vert ^{2}-1\leq \frac{10}{r}\text{ \ on }%
M\setminus B_{p}\left( R_{0}\right),  \label{d2}
\end{equation}%
where $R_{0}=5r_{0}^{2}.$ In particular, $|\nabla \rho|\leq 1+\frac{5}{r}$ on $M\setminus B_{p}\left( R_{0}\right).$ 
Integrating this estimate along minimizing geodesics, we conclude that

\begin{equation*}
\rho \leq r+6\ln r.
\end{equation*}%
Finally, (\ref{d2}) implies that 
\begin{equation*}
\left\vert \nabla \rho \right\vert ^{2}\left( 1-\frac{10}{r}\right) \leq
\left( 1+\frac{10}{r}\right) \left( 1-\frac{10}{r}\right) <1.
\end{equation*}%
This means that 
\begin{equation*}
\left\vert \nabla \rho \right\vert ^{2}-1\leq \frac{10}{r}\left\vert \nabla
\rho \right\vert ^{2}
\end{equation*}%
as claimed. The proposition is proved.
\end{proof}

We remark that the linear decay rate in the gradient estimate of Proposition \ref{rho} is sharp.
This is because $\rho =r+\ln r$ for $M=\mathbb{R}^{3}$ equipped with the Euclidean metric, 
since $\Delta h=h$ for $h=e^{-\left( r+\ln r\right) }.$ 

Another fact is that there exists a constant $a>0$ such that 
\begin{equation}
\Delta \rho ^{-a}\geq 0\text{ \ on }M\setminus B_{p}\left( R_{0}\right) .
\label{S}
\end{equation}%
Indeed, by Proposition \ref{rho} we have
\begin{eqnarray*}
\Delta \rho ^{-a} &=&-a\left( \Delta \rho \right) \rho ^{-a-1}+a\left(
a+1\right) \left\vert \nabla \rho \right\vert ^{2}\rho ^{-a-2} \\
&=&a\left( 1-\left\vert \nabla \rho \right\vert ^{2}\right) \rho
^{-a-1}+a\left( a+1\right) \left\vert \nabla \rho \right\vert ^{2}\rho
^{-a-2} \\
&\geq &-10a\left\vert \nabla \rho \right\vert ^{2}\rho ^{-a-2}+a\left(
a+1\right) \left\vert \nabla \rho \right\vert ^{2}\rho ^{-a-2}.
\end{eqnarray*}%
Hence, for $a=9$, we get  $\Delta \rho ^{-a}\geq 0$ as claimed.

Finally, by extending $\rho$ to $B_p(1)$ by setting $\rho=0$ there, the set $\{\rho<s\}$ is connected 
for every $s>0.$ Otherwise, there would exist a connected component $\Omega_s$ of $\{\rho<s\}$ 
with $\rho=s$ on the boundary $\partial \Omega_s$, and such that $\Omega_s \cap B_p(1)=\emptyset$.  Applying
the maximum principle to equation $\Delta \rho=-1+|\nabla \rho|^2$ at the minimum point of $\rho$ in $\Omega_s,$
one arrives at an obvious contradiction.

\section{Size of fingers}\label{3}

In this section, we consider level sets $\left\{ \rho=s\right\} $ 
of the function $\rho$ constructed in the previous section. For  $t_0= 2R_0$, with $R_0$ as in Proposition \ref{rho},  we have that  $\{\rho>t_0\}\subset M\setminus B_p(R_0).$

As $M$ has only one end, there is only one unbounded connected component of $M\setminus \{\rho\leq s\}.$ Denote it with
$M_{s}$ and let $\ell_{0}\left( s\right) $ be its boundary, i.e., 
\begin{equation}
\ell_{0}\left( s\right) :=\{\rho= s\} \cap \overline{M_{s}}=\partial M_{s}.
\label{a2}
\end{equation}%
%We mainly consider $s>t_0$. 
Since the Ricci curvature of $M$ is nonnegative, by Liu's work \cite{Liu}, we may assume without loss of generality that  $M$ is diffeomorphic to $\mathbb{R}^3$.
Otherwise, \cite{Liu} implies that the universal cover of $M$ is isometric to $\mathbb{R}\times N$, for a complete surface $N$  of non-negative sectional curvature, and Theorem \ref{B} is valid in that case by Cohn-Vossen. 
Hence, according to \cite {LT, MW2}, $\ell_0(s)$ is connected for all $s>t_0.$ 
 
\begin{lemma}
\label{l} Let $\left( M^{3},g\right) $ be a complete noncompact three-dimensional  manifold with 
$\mathrm{Ric}\geq 0$. Then $\ell_{0}\left( s\right)$ is connected for any $s>t_0$.
\end{lemma}

% For $t_0$ large enough to be specified below,  
Denote with
\begin{equation}\label{level}
L(t_0,s):=\{\rho<s\}\cap M_{t_0}\;\;\text{and}\;\;\ell(s):=\{\rho=s\}\cap M_{t_0}
\end{equation}
and let 
\begin{equation}
\ell_{1}\left( s\right) :=\ell\left( s\right) \setminus \ell_{0}\left( s\right)
=\ell\left( s\right) \setminus \overline{M_{s}}. \label{a3}
\end{equation}
Since $\{\rho<s\}$ is connected, it implies that $\ell_1(s)$ is the boundary of
bounded connected components of $M\setminus \{\rho<s\}$.  Note that $\ell(s)=\{\rho=s\}$ for $s$  large enough,  
but for $s$ close to $t_0$ there may exist connected components of $\{\rho=s\}$ that are not contained in $M_{t_0}$. 

We aim to establish the following proposition.

\begin{proposition}
\label{Finger_rho}Let $\left( M^{3},g\right) $ be a complete noncompact
three-dimensional manifold with $\mathrm{Ric}\geq 0$ on $M$ and scalar curvature $S\geq 1$ on $B_{p}\left( 2R\right).$ 
Then there exists a constant $C>0$ such that if $\ell_{1}\left( t\right) \neq \emptyset $ for some $t_{0}<t<2R,$ then 
\begin{equation*}
\ell_{1}\left( t\right) \subset M_{t-C\ln t}.
\end{equation*}
\end{proposition}

To prepare for the proof, let us first recall the following result from  \cite{CLS}  (see also \cite{CL}). 
Here, $N_{t}$ is the unbounded connected component of $M\setminus \overline{B_{p}\left( t\right) }.$

\begin{lemma}
\label{bubble}Let $\left( M^{3},g\right) $ be a complete noncompact
three-dimensional manifold with scalar curvature $S\geq 1$ on $B_{p}\left(
2R\right).$ There exist universal constants $L, C_{0}>0$ and a connected smooth surface $\Sigma _{\tau}
\subset \left(N_{\tau-L}\setminus N_{\tau-1}\right)$ for each $\tau \in (2L,2R)$, 
such that $\Sigma _{\tau}$ separates $\partial N_{\tau-L}$ from $\partial
N_{\tau-1}$ and its diameter $\mathrm{diam}\left( \Sigma _{\tau}\right) \leq C_{0}.$
\end{lemma}

Such a surface $\Sigma _{\tau}$ is constructed by minimizing a suitable functional and is commonly referred to as a $\mu$-bubble, see \cite{G}.
As in \cite{CLS}, we use this diameter estimate to bound the size of $N_{\tau}\setminus N_{\tau+L}$, but with a different approach.

We may assume that $R_0>5(L+C_0).$ We first establish a result analogous to Proposition \ref{Finger_rho} for geodesic spheres. 
\begin{lemma}
\label{Finger}Let $\left( M^{3},g\right) $ be a complete noncompact
three-dimensional manifold with $\mathrm{Ric}\geq 0$ on $M$ and scalar curvature $S\geq 1$ on $B_{p}\left( 2R\right).$ 
Then there exists a constant $C>0$ such that
\begin{equation*}
\left( \partial B_{p}\left( t\right)\cap N_{R_0}\right) \setminus N_{t}\subset N_{t-\frac{C}{\sqrt{t}}}
\end{equation*}
whenever $\left(\partial B_{p}\left( t\right)\cap N_{R_0}\right) \setminus N_{t}\neq \emptyset$ for 
$R_{0}<t<2R.$
\end{lemma}

\begin{proof}
We first claim that if there exists $T_0>0$ such that  both
\begin{equation*}
\Gamma _{0}:=\left( \partial B_{p}\left( \tau\right) \cap N_{\tau-1}\right)
\setminus N_{\tau}\neq \emptyset
\end{equation*}
and
\begin{equation}
\Gamma _{T_{0}}:=\left( \partial B_{p}\left( \tau+T_{0}\right) \cap
N_{\tau-1}\right) \setminus N_{\tau}\neq \emptyset \label{n2}
\end{equation}
for some $R_0\leq \tau<2R$, then  
\begin{equation}\label{T0bound}
T_{0}\leq \frac{C}{\sqrt{\tau}}.
\end{equation}

Indeed, let $\gamma :\left[ 0,\infty \right) \rightarrow M$ be a normal geodesic ray
starting at $\gamma \left( 0\right) =p$ and let $q:=\gamma \left( 3\tau\right).$
As in Abresch-Gromoll \cite{AG}, define the excess function by 
\begin{equation*}
E\left( x\right) :=d\left( p,x\right) +d\left( q,x\right) -d\left(p,q\right).
\end{equation*}%
The triangle inequality implies that $E\geq 0$ on $M$ and $E\left( \gamma
\left( t\right) \right) =0$ for all $t<3\tau.$ As $\mathrm{Ric}\geq 0,$
the Laplacian comparison theorem implies that 
\begin{equation}
\Delta E\left( x\right) \leq \frac{3}{\tau-C}  \label{n1}
\end{equation}%
in the weak sense for any $x\in N_{\tau-C}\setminus N_{\tau+C}.$ Observe that for any $0<T<\min \left\{ 1,T_{0}\right\},$
\begin{equation*}
\Gamma _{T}:=\left( \partial B_{p}\left( \tau+T\right) \cap N_{\tau-1}\right)
\setminus N_{\tau}\neq \emptyset
\end{equation*}%
and that $\Gamma _{T}$ separates $\Gamma _{0}$ from $\Gamma _{T_{0}}.$ 
So for $z\in \Gamma _{T}$, we have $d\left( z,\Gamma _{0}\right) =T$. In fact,
this distance is realized by the segment in $M\setminus B_{p}\left( \tau\right)$ of
a minimizing geodesic from $p$ to $z.$ 
Moreover, as $q\in N_{\tau},$ 
we have 
\begin{eqnarray*}
d\left( z,q\right)  &\geq &d\left( z,\Gamma _{0}\right) +d\left( q,\partial
N_{\tau}\right)  \\
&=&T+2\tau.
\end{eqnarray*}%
Indeed, the second line is because $\gamma $ is a distance minimizing path 
from $q=\gamma(3\tau)$ to $\partial N_{\tau}\subset \partial B_{p}\left( \tau\right).$ 
Since $d\left( p,z\right) =\tau+T$ and $d\left( p,q\right) =3\tau,$ this implies
that 
\begin{equation}
E\left( z\right) \geq 2T  \label{n3}
\end{equation}%
for any $z\in \Gamma _{T}.$ We now estimate $d\left( z,\gamma \right),$
which is where the $\mu $-bubble diameter estimate Lemma \ref{bubble} is used. 

First, note that 
\begin{equation}
d\left( z,\gamma \right) \geq d\left( z,\Gamma _{0}\right) =T.
\label{n30}
\end{equation}
For an upper bound of $d\left( z,\gamma \right),$ recall that by Lemma \ref%
{bubble} there exists a separating connected surface $\Sigma _{\tau}\subset
N_{\tau-L}\setminus N_{\tau-1}$ with $\mathrm{diam}\left( \Sigma _{\tau}\right) \leq C_{0}.$ 
For any $z\in \Gamma _{T}\subset N_{\tau-1}$ it follows that $\Sigma_\tau$ separates $p$ and $z$ and  $\Sigma_\tau$ intersects $\gamma$, therefore
\begin{equation}
d\left( z,\gamma \right) \leq d\left( z,\Sigma _{\tau}\right) +\mathrm{diam}\left(
\Sigma _{\tau}\right).  \label{n4}
\end{equation}%
Moreover, any minimizing geodesic from $p$
to $z$ must intersect $\Sigma _{\tau}$ at a point $z^{\prime }\in \Sigma
_{\tau}$ and $d\left( p,z^{\prime }\right) \geq \tau-L.$ Therefore, 
\begin{eqnarray*}
d\left( z,\Sigma _{\tau}\right)  &\leq &d\left( z,z^{\prime }\right)  \\
&=&d\left( p,z\right) -d\left( p,z^{\prime }\right)  \\
&\leq &T+L.
\end{eqnarray*}%
It follows by (\ref{n4}) that 
$$
d\left( z,\gamma \right) \leq T+L+C_{0}.
$$
Together with (\ref{n30}), we conclude that
\begin{equation}
T\leq d\left( z,\gamma \right) \leq T+L+C_{0}.  \label{n40}
\end{equation}%
Denote with 
\begin{equation}
\lambda :=T+L+C_{0}+1.  \label{n5}
\end{equation}%
Since $E\left( \gamma \right) =0,$ by (\ref{n40}) there
exists $y\in B_{z}\left( \lambda -1\right) \setminus B_{z}\left( T\right) $
such that 
\begin{equation}
E\left( y\right) =0.  \label{E=0}
\end{equation}
Moreover, since $T<1$, we have that $B_z(\lambda)\subset N_{\tau-\lambda}\setminus N_{\tau+\lambda+1}$.
From  (\ref{n1}) and the fact that $\tau\geq R_0>4\lambda$ we obtain  that 
\begin{equation}
\Delta E\leq \frac{4}{\tau}\text{ \ on }B_{z}\left( \lambda \right).
\label{n6}
\end{equation}

Following Abresch-Gromoll \cite{AG}, define the function 
\begin{equation*}
\varphi \left( d\right) :=\frac{1}{3}\frac{\lambda ^{3}}{d}+\frac{1}{6}d^{2}-%
\frac{1}{2}\lambda ^{2}.
\end{equation*}
It satisfies the following properties for all $d\in \left( 0,\lambda\right)$:
\begin{eqnarray*}
\varphi \left( \lambda \right)  &=&0,\text{ } \\
\text{ }\varphi \left( d\right)  &>&0, \\
\varphi ^{\prime }\left( d\right)  &\leq &0, \\
\varphi ^{\prime \prime }\left( d\right) +\frac{2}{d}\varphi ^{\prime
}\left( d\right)  &=&1.
\end{eqnarray*}%
We view $\varphi $ as a function on $B_{z}\left( \lambda \right) $ by 
$\varphi \left( x\right) =\varphi \left( d\left( z,x\right) \right).$ The
Laplacian comparison theorem implies that 
\begin{eqnarray}
\Delta \varphi  &=&\varphi ^{\prime }\Delta d+\varphi ^{\prime \prime
}\left\vert \nabla d\right\vert ^{2}  \label{n7} \\
&\geq &\varphi ^{\prime \prime }\left( d\right) +\frac{2}{d}\varphi ^{\prime
}\left( d\right)   \notag \\
&=&1  \notag
\end{eqnarray} on $B_z(\lambda).$ 
Consequently, by (\ref{n6}) and (\ref{n7}), the function 
\begin{equation*}
w\left( x\right) :=E\left( x\right) -\frac{4}{\tau}\varphi \left( x\right) 
\end{equation*}%
satisfies $\Delta w\leq 0$ on $B_{z}\left( \lambda \right).$ By the maximum
principle, on $B_{z}\left( \lambda \right) \setminus B_{z}\left( 
\frac{T}{2}\right),$ $w$ achieves its minimum at some $x_{0}\in
\partial B_{z}\left( \frac{T}{2}\right) \cup \partial B_{z}\left( \lambda
\right).$ According to (\ref{E=0}), there exists $y\in B_{z}\left( \lambda -1\right)
\setminus B_{z}\left( T\right) $ such that $w\left( y\right) <0.$ Therefore,
$w\left( x_{0}\right) <0$ as well. Since $w\geq 0$ on 
$\partial B_{z}\left( \lambda\right) $ by $\varphi \left(\lambda \right) =0,$ 
we conclude that $x_{0}\in \partial B_{z}\left( \frac{T}{2}\right).$ 

As $\varphi \left(
x_{0}\right) =\frac{2}{3}\frac{\lambda ^{3}}{T}+\frac{1}{24}T^{2}-\frac{1}{2}%
\lambda ^{2}$, this means that 
\begin{equation*}
E\left( x_{0}\right) <\frac{1}{\tau}\left( \frac{8}{3}\frac{\lambda ^{3}}{T}+%
\frac{1}{6}T^{2}-2\lambda ^{2}\right) .
\end{equation*}%
Since $\left\vert \nabla E\right\vert \leq 2$, it follows that 
\begin{eqnarray*}
E\left( z\right)  &\leq &E\left( x_{0}\right) +T \\
&<&T+\frac{1}{\tau}\left( \frac{8}{3}\frac{\lambda ^{3}}{T}+\frac{1}{6}%
T^{2}-2\lambda ^{2}\right) .
\end{eqnarray*}%
Together with (\ref{n3}), this implies that 
\begin{equation}
T<\frac{1}{\tau}\left( \frac{8}{3}\frac{\lambda ^{3}}{T}+\frac{1}{6}%
T^{2}-2\lambda ^{2}\right) ,  \label{n8}
\end{equation}%
where $\lambda $ is given by (\ref{n5}). Note that (\ref{n8}) holds for any 
$0<T<\min \left\{ 1,T_{0}\right\}.$ This immediately implies that 
$T_{0}\leq $ $\frac{2\lambda ^{2}}{\sqrt{\tau}}$ as claimed in (\ref{T0bound}).

To complete the proof, now assume
\begin{equation*}
\left( \partial B_{p}\left( t\right)\cap N_{R_0}\right) \setminus N_{t}\neq \emptyset 
\end{equation*}
for some $R_{0}<t<2R.$ 
Let 
\begin{equation*}
\tau:=\sup \left\{ s\; \big|\; s <t\text{ \ and }\left(\partial B_{p}\left( t\right)\cap N_{R_0}\right)
\setminus N_{t}\subset N_{s }\right\}.
\end{equation*}
Then  $\tau\geq R_0$ and $\left(\partial B_{p}\left( t\right)\cap N_{R_0}\right) \setminus N_{t}$ is no longer a subset of $N_{\tau}.$ In particular, as $\left(\partial B_p(t)\cap N_{R_0}\right)\subset N_{\tau-1}$, we have $\left(\partial B_p(t)\cap N_{\tau-1} \right)\setminus N_{\tau}\neq \emptyset$.  
To see that $\left(\partial B_p(\tau)\cap N_{\tau-1} \right)\setminus N_{\tau}\neq \emptyset$ as well,
note that by the definition of $\tau,$ there exists a finger $F$ of $M\setminus B_p(\tau)$ 
such that the intersection $F\cap \left(M\setminus B_p(t)\right)$ is a finger of $M\setminus B_p(t).$  
This implies that $F\subset N_{\tau-1}$ and $\left(\partial B_p(\tau)\cap N_{\tau-1} \right)\setminus N_{\tau}\neq \emptyset.$
So according to the claim, we must have $t-\tau\leq \frac{C}{\sqrt{\tau}}.$ Equivalently, $\tau\geq t-\frac{C}{\sqrt{t}}.$ Therefore 
\begin{equation*}
\left( \partial B_{p}\left( t\right) \cap N_{R_0} \right)\setminus N_{t} \subset N_{t-\frac{C}{\sqrt{t}}}.
\end{equation*}
This proves the result.
\end{proof}

In \cite{CLS},  %and \cite{X},
based on the Cheeger-Colding almost splitting theorem, it was shown that fingers must 
lie between annuli of fixed size. Our approach here is different and seems to yield a somewhat better estimate.
We now prove Proposition {\ref {Finger_rho}.

\begin{proof}[Proof of Proposition \protect\ref{Finger_rho}]
Lemma \ref{rho} implies that there exists constant $C_{0}>0$ such that 
\begin{eqnarray}
r-C_{0}\ln r \leq \rho \leq r+C_{0}\ln r,  \label{n10} \\
\rho -C_{0}\ln \rho  \leq r\leq \rho +C_{0}\ln \rho .  \notag
\end{eqnarray}%
Therefore,  
\begin{equation}
N_{t-2C_{0}\ln t}\subset M_{t-3C_{0}\ln t} \;\;\;\;\text{and}\;\;\; M_{t_0}\subset N_{R_0}.  \label{n11}
\end{equation}%
Assume by contradiction that $\ell_{1}\left( t\right) $ is not a subset of $%
M_{t-3C_{0}\ln t}$. Then (\ref{n11}) implies that $\ell_{1}\left( t\right) $ is
not a subset of $N_{t-2C_{0}\ln t}$ either as it lies in bounded connected components 
of $M\setminus L\left( t_0,t-3C_{0}\ln t\right).$ Let $\eta $ be a minimizing geodesic from $p$ to $z\in
\ell_{1}\left( t\right).$ By (\ref{n10}) we have that $r\left( z\right) \geq
t-C_{0}\ln t.$ Therefore, $\eta $ intersects $\partial B_{p}\left( t-C_{0}\ln
t\right).$ It follows that for  $s:=t-C_{0}\ln t,$ the set $\left(\partial
B_{p}\left( s\right)\cap N_{R_0}\right) \setminus N_{s}$ is not empty and it is not a subset of $N_{s-C_{0}\ln s}.$  
This contradicts with Lemma \ref{Finger}.
\end{proof}

\section{Proof of the main result}\label{4}

With the preparations in the previous sections, we are now ready to prove Theorem \ref{B}. 
We continue to adopt the same notations as before.

\begin{lemma}
\label{Ric}Let $\left( M^{3},g\right) $ be a three-dimensional complete
Riemannian manifold with $\mathrm{Ric}\geq 0.$ Then on each regular 
level set $\ell(s)$ of $\rho$ with $s>t_{0},$
 we have 
\begin{equation*}
\int_{\ell_{0}\left( s\right) }\left( \left\vert \nabla ^{2}\rho \right\vert
^{2}-\left\vert \nabla \left\vert \nabla \rho \right\vert \right\vert ^{2}+%
\mathrm{Ric}\left( \nabla \rho ,\nabla \rho \right) \right) \left\vert
\nabla \rho \right\vert ^{-2}\geq \frac{1}{2}\int_{\ell_{0}\left( s\right)
}S-4\pi. 
\end{equation*}%
\end{lemma}

\begin{proof}
According to Lemma \ref{RicS}, 

\begin{eqnarray*}
\mathrm{Ric}\left( \nabla \rho ,\nabla \rho \right) \left\vert \nabla \rho
\right\vert ^{-2} &=&\frac{1}{2}S-\frac{1}{2}\bar{S}+\frac{1}{2}\frac{1}{%
\left\vert \nabla \rho \right\vert ^{2}}\left( \left\vert \nabla \left\vert
\nabla \rho \right\vert \right\vert ^{2}-\left\vert \nabla ^{2}\rho
\right\vert ^{2}\right) \\
&&+\frac{1}{2}\frac{1}{\left\vert \nabla \rho \right\vert ^{2}}\left( \left(
\Delta \rho \right) ^{2}-2\frac{\left\langle \nabla \left\vert \nabla \rho
\right\vert ,\nabla \rho \right\rangle }{\left\vert \nabla \rho \right\vert }%
\Delta \rho +\left\vert \nabla \left\vert \nabla \rho \right\vert
\right\vert ^{2}\right) ,
\end{eqnarray*}%
where $\bar{S}$ denotes the scalar curvature of the surface $\{\rho =s\}.$
Therefore, 
\begin{equation*}
\mathrm{Ric}\left( \nabla \rho ,\nabla \rho \right) \left\vert \nabla \rho
\right\vert ^{-2}\geq \frac{1}{2}S-\frac{1}{2}\bar{S}+\frac{1}{2}\frac{1}{%
\left\vert \nabla \rho \right\vert ^{2}}\left( \left\vert \nabla \left\vert
\nabla \rho \right\vert \right\vert ^{2}-\left\vert \nabla ^{2}\rho
\right\vert ^{2}\right) .
\end{equation*}%
Consequently, for any $s>t_0$, 
\begin{eqnarray*}
&&\int_{\ell_{0}\left( s\right) }\left( \left\vert \nabla ^{2}\rho \right\vert
^{2}-\left\vert \nabla \left\vert \nabla \rho \right\vert \right\vert ^{2}+%
\mathrm{Ric}\left( \nabla \rho ,\nabla \rho \right) \right) \left\vert
\nabla \rho \right\vert ^{-2} \\
&\geq &\frac{1}{2}\int_{\ell_{0}\left( s\right) }S-\frac{1}{2}\int_{\ell_{0}\left(
s\right) }\bar{S} \\
&\geq&\frac{1}{2}\int_{\ell_{0}\left( s\right) }S-4\pi .
\end{eqnarray*}%
The last line is due to the Gauss-Bonnet theorem and Lemma \ref{l}.
\end{proof}

To estimate the integral on $\ell_{1}\left( s\right),$ we denote with 
\begin{equation}\label{omega}
\omega(t):=\int_{1}^{t}\mathrm{V}\left( B_{p}\left(
r\right) \right) \frac{\ln r}{r^{2}}dr.
\end{equation}
Note that by volume comparison estimates we have 
$$
\int_t^{2t} \mathrm{V}\left( B_{p}\left(
r\right) \right) \frac{\ln r}{r^{2}}dr\leq C\int_{\frac{t}{2}}^t \mathrm{V}\left( B_{p}\left(
r\right) \right) \frac{\ln r}{r^{2}}dr.
$$
Therefore, it follows that 
\begin{equation}\label{omega2t}
\omega(2t)\leq C\omega (t).
\end{equation} 

\begin{lemma}
\label{l1}Let $\left( M^{3},g\right) $ be a complete noncompact three
dimensional manifold with $\mathrm{Ric}\geq 0$ and scalar curvature $S\geq 1$ 
on $B_{p}\left( 2R\right).$ Then there exists  a constant $C>0$ such that

\begin{equation*}
\int_{t_0}^{t}\left( \int_{\ell_{1}\left( s\right) }\left\vert \nabla \rho
\right\vert \right) ds\leq C\omega(t)
\end{equation*}
for all $t_{0}<t<2R.$
\end{lemma}

\begin{proof}
In view of (\ref{omega2t}), we may assume that $t>2t_0$. For $t_0<s<t,$ let $M_{s}$ be the unbounded connected component of 
$M\setminus \overline{L\left( t_0,s\right) }.$ If $\ell_{1}\left( s\right)\neq \emptyset,$ 
then Proposition \ref{Finger_rho} implies that 
\begin{equation*}
\partial \left( M_{s-C\ln s}\cap L(t_0,s) \right)
=\ell_{0}\left( s-C\ln s\right) \cup \ell_{0}\left( s\right) \cup \ell_{1}\left(s\right).
\end{equation*}
According to (\ref{S}), there exists $a>0$ such that $\Delta \rho ^{-a}\geq 0.$ It follows that

\begin{eqnarray*}
0 &\leq &\int_{M_{s-C\ln s}\cap L(t_0,s) }\Delta \rho ^{-a} \\
&=&\int_{\ell_{0}\left( s\right) }\left( \rho ^{-a}\right) _{\nu
}+\int_{\ell_{1}\left( s\right) }\left( \rho ^{-a}\right) _{\nu
}-\int_{\ell_{0}\left( s-C\ln s\right) }\left( \rho ^{-a}\right) _{\nu },
\end{eqnarray*}%
where $\nu =\frac{\nabla \rho }{\left\vert \nabla \rho \right\vert }.$
Consequently, 
\begin{equation*}
\frac{1}{s^{a+1}}\int_{\ell_{1}\left( s\right) }\left\vert \nabla \rho
\right\vert \leq \frac{1}{\left( s-C\ln s\right) ^{a+1}}\int_{\ell_{0}\left(
s-C\ln s\right) }\left\vert \nabla \rho \right\vert -\frac{1}{s^{a+1}}%
\int_{\ell_{0}\left( s\right) }\left\vert \nabla \rho \right\vert.
\end{equation*}%
It follows that 
\begin{eqnarray}
\int_{\ell_{1}\left( s\right) }\left\vert \nabla \rho \right\vert  &\leq
&\int_{\ell_{0}\left( s-C\ln s\right) }\left\vert \nabla \rho \right\vert
-\int_{\ell_{0}\left( s\right) }\left\vert \nabla \rho \right\vert   \label{a8}
\\
&&+\frac{C\ln s}{s}\int_{\ell_{0}\left( s-C\ln s\right) }\left\vert \nabla \rho
\right\vert.  \notag
\end{eqnarray}
Integrating (\ref{a8}) with respect to $s$ from $t_{0}$ to $t$ we obtain that 
\begin{eqnarray}
\;\;\;\;\;\;\;\;\;\;\int_{t_{0}}^{t}\left( \int_{\ell_{1}\left( s\right) }\left\vert \nabla \rho
\right\vert \right) ds &\leq &\int_{t_{0}}^{t}\left( \int_{\ell_{0}\left(
s-C\ln s\right) }\left\vert \nabla \rho \right\vert \right)
ds-\int_{t_{0}}^{t}\left( \int_{\ell_{0}\left( s\right) }\left\vert \nabla \rho
\right\vert \right) ds  \label{a9} \\
&&+C\int_{t_{0}}^{t}\frac{\ln s}{s}\left( \int_{\ell_{0}\left( s-C\ln s\right)
}\left\vert \nabla \rho \right\vert \right) ds.  \notag
\end{eqnarray}%
Note that by making the substitution $r=s-C\ln s$ we get%
\begin{eqnarray*}
\int_{t_{0}}^{t}\left( \int_{\ell_{0}\left( s-C\ln s\right) }\left\vert \nabla
\rho \right\vert \right) \left( 1-\frac{C}{s}\right) ds &=&\int_{t_{0}-C\ln
t_{0}}^{t-C\ln t}\left( \int_{\ell_{0}\left( r\right) }\left\vert \nabla \rho
\right\vert \right) dr \\
&\leq &\int_{t_{0}}^{t}\left( \int_{\ell_{0}\left( r\right) }\left\vert \nabla
\rho \right\vert \right) dr\\
&&+\int_{t_{0}-C\ln
t_{0}}^{t_0}\left( \int_{\ell_{0}\left( r\right) }\left\vert \nabla \rho
\right\vert \right) dr.
\end{eqnarray*}
The second integral on the right side can be estimated by using Proposition \ref{rho} and the volume comparison theorem,
\begin{eqnarray*}
\int_{t_{0}-C\ln
t_{0}}^{t_0}\left( \int_{\ell_{0}\left( r\right) }\left\vert \nabla \rho
\right\vert \right) dr&\leq& C\mathrm{V}\left( B_p(t_0+C\ln t_0)\setminus B_p(t_0-C\ln t_0)\right)\\
&\leq& C\mathrm{V}\left(B_p(t_0)\right)\frac{\ln t_0}{t_0} \\ 
&\leq & C \omega(t)
\end{eqnarray*}
for any $t>2t_0.$ Hence, we get that
\begin{eqnarray*}
\int_{t_{0}}^{t}\left( \int_{\ell_{0}\left( s-C\ln s\right) }\left\vert \nabla
\rho \right\vert \right) ds &\leq &\int_{t_{0}}^{t}\left( \int_{\ell_{0}\left(
s\right) }\left\vert \nabla \rho \right\vert \right) ds+C\omega(t) \\
&&+C\int_{t_{0}}^{t}\frac{1}{s}\left( \int_{\ell_{0}\left( s-C\ln s\right)
}\left\vert \nabla \rho \right\vert \right) ds.
\end{eqnarray*}%
It follows from (\ref{a9}) that 
\begin{equation*}
\int_{t_{0}}^{t}\left( \int_{\ell_{1}\left( s\right) }\left\vert \nabla \rho
\right\vert \right) ds\leq C\omega(t)+C\int_{t_{0}}^{t}\frac{\ln s}{s}\left(
\int_{\ell_{0}\left( s-C\ln s\right) }\left\vert \nabla \rho \right\vert
\right) ds.
\end{equation*}%
Moreover, using the same substitution as above, we have
\begin{eqnarray*}
\int_{t_{0}}^{t}\frac{\ln s}{s}\left( \int_{\ell_{0}\left( s-C\ln s\right)
}\left\vert \nabla \rho \right\vert \right) ds &\leq &\int_{t_{0}}^{t}\frac{%
\ln s}{s}\left( \int_{\ell(s-C\ln s)}\left\vert \nabla \rho
\right\vert \right) ds \\
&\leq &C\omega(t)+C\int_{t_{0}}^{t-C\ln t}\frac{\ln r}{r}\left( \int_{\ell(r)}\left\vert \nabla \rho \right\vert \right) dr \\
&=&C\omega(t)+C\int_{L(t_0, t-C\ln t)}\frac{\ln \rho }{\rho }%
\left\vert \nabla \rho \right\vert ^{2}.
\end{eqnarray*}%
Therefore, in view of Lemma \ref{rho}, we get that 
\begin{equation}
\int_{t_{0}}^{t}\left( \int_{\ell_{1}\left( s\right) }\left\vert \nabla \rho
\right\vert \right) ds\leq C\omega(t)+C\int_{B_{p}\left( t\right)\setminus B_p(1) }\frac{\ln \rho }{%
\rho }.  \label{a10}
\end{equation}%
As $\mathrm{Ric}\geq 0,$ 
\begin{eqnarray*}
\int_{B_{p}\left( t\right)\setminus B_p(1) }\frac{\ln \rho }{\rho } &=&\int_{1}^{t}%
\frac{\ln r}{r}\mathrm{A}\left( \partial B_{p}\left( r\right) \right) dr \\
&\leq &C\int_{1}^{t}\mathrm{V}\left( B_{p}\left( r\right) \right) \frac{%
\ln r}{r^{2}}dr.
\end{eqnarray*}
In conclusion, (\ref{a10}) implies that 
\begin{equation*}
\int_{t_{0}}^{t}\left( \int_{\ell_{1}\left( s\right) }\left\vert \nabla \rho
\right\vert \right) ds\leq C\omega(t)
\end{equation*}
for all $2t_0<t<2R.$ The lemma is proved.
\end{proof}

\begin{lemma}
\label{Area}Let $\left( M^3,g\right) $ be a complete noncompact
three dimensional manifold with $\mathrm{Ric}\geq 0.$ Then there exists constant
$C>0$ such that
 
\begin{equation*}
\int_{t_0}^{t}\frac{1}{s}\mathrm{A}\left( \left\{ \rho =s\right\} \right)
ds\leq C \omega(t)
\end{equation*}%
for all $t>2t_{0}.$
\end{lemma}

\begin{proof}
We may assume that $t_0, t \in \mathbb{N}$. Then 
\begin{eqnarray*}
\int_{t_0}^{t}\frac{1}{s}\mathrm{A}\left( \left\{ \rho =s\right\} \right) ds
&=&\sum_{k=t_0+1}^{ t}\int_{k-1}^{k}\frac{1}{s}\mathrm{A}\left( \left\{
\rho =s\right\} \right) ds \\
&\leq &C\sum_{k=t_0+1}^{ t}\frac{1}{k}\int_{k-1}^{k}\mathrm{A}\left(
\left\{ \rho =s\right\} \right) ds \\
&=&C\sum_{k=t_0+1}^{ t}\frac{1}{k}\int_{\left\{ k-1<\rho <k\right\}
}\left\vert \nabla \rho \right\vert .
\end{eqnarray*}%
Proposition \ref{rho} implies that 
\begin{equation*}
\left\{ k-1<\rho <k\right\} \subset B_{p}\left( k+C\ln k\right) \setminus
B_{p}\left( k-C\ln k\right) .
\end{equation*}%
By the volume comparison theorem, 
\begin{equation*}
\mathrm{V}\left( B_{p}\left( k+C\ln k\right) \setminus B_{p}\left( k-C\ln
k\right) \right) \leq C\mathrm{V}\left( B_{p}\left( k\right) \setminus
B_{p}\left( k-1\right) \right) \ln k.
\end{equation*}%
Therefore, 
\begin{eqnarray*}
\int_{t_0}^{t}\frac{1}{s}\mathrm{A}\left( \left\{ \rho =s\right\} \right)
ds &\leq &C\sum_{k=t_0+1}^{ t}\frac{\ln k}{k}\mathrm{V}\left( B_{p}\left(
k\right) \setminus B_{p}\left( k-1\right) \right)  \\
&\leq &C\int_{t_{0}}^{t}\mathrm{A}\left( \partial B_{p}\left( s\right)
\right) \frac{\ln s}{s}ds \\
&\leq &C\int_{t_{0}}^{t}\mathrm{V}\left( B_{p}\left( s\right) \right) \frac{%
\ln s}{s^{2}}ds.
\end{eqnarray*}
\end{proof}

Let us denote with 
\begin{equation}\label{Ct0}
\beta:=\max\left\{-  \int_{\ell(t_0)}\vert \nabla \rho\vert_{\nu}\,, \int_{\ell(t_0)}\vert\nabla \rho\vert \right\}.
\end{equation}
We note that  
 \begin{equation}\label{Ct0 bound}
 \beta\leq C \mathrm{V}(M\setminus M_{t_0}),
 \end{equation}
where $M_{t_0}$ denotes the only unbounded connected component of $M\setminus \{\rho \leq t_0\}$. 

 To see this, let $\phi (t)$ be a smooth cut-off function  such that 
$\phi(t)=0$ for $t\leq t_0-1$ and $\phi(t)=1$ for $t\geq t_0$.  The function $\phi(x):=\phi(\rho(x))$ satisfies 
\begin{equation}\label{A1}
|\nabla \phi |+|\Delta \phi |\leq C. 
\end{equation}
 
Then
\begin{eqnarray}\label{A2}
\int_{\ell(t_0)}|\nabla \rho |=\int_{M\setminus M_{t_0}}\left( \phi^2\Delta \rho+\left<\nabla \rho,\nabla \phi^2\right>\right )\leq C\mathrm{V}(M\setminus M_{t_0}).
\end{eqnarray}
 Similarly, we have 
 \begin{eqnarray*}
\int_{\ell(t_0)}\vert \nabla \rho\vert_{\nu}=\int_{M\setminus M_{t_0}}\left(\phi^2 \Delta \left\vert \nabla \rho \right\vert +
\left<\nabla \vert \nabla \rho\vert, \nabla \phi^2\right> \right )
\end{eqnarray*}
The Bochner formula for $\rho$ satisfying $\Delta \rho=\vert \nabla \rho\vert^2 -1$ implies that 
 \begin{equation*}
\Delta \left\vert \nabla \rho \right\vert \geq 2\left\langle \nabla \vert \nabla \rho\vert ,\nabla \rho \right\rangle,
\end{equation*}
therefore
  \begin{eqnarray*}
-\int_{\ell(t_0)}\vert \nabla \rho\vert_{\nu}\leq -2\int_{M\setminus M_{t_0}}\phi^2 \left\langle \nabla \vert \nabla \rho\vert ,\nabla \rho \right\rangle - \int_{M\setminus M_{t_0}} \left<\nabla \vert \nabla \rho\vert, \nabla \phi^2\right> .
\end{eqnarray*}
Integrating by parts the terms on the right side and using (\ref{A1}) and (\ref{A2}) we obtain that  
$$-\int_{\ell(t_0)}\vert \nabla \rho\vert_{\nu} \leq C\mathrm{V}(M\setminus M_{t_0}).$$

We now come to the proof of Theorem \ref{B} which is restated below. 

\begin{theorem}
\label{B1}Let $\left( M^3,g\right) $ be a complete noncompact three-dimensional
manifold with $\mathrm{Ric}\geq 0.$ Assume moreover that the scalar curvature $S$ of
$\left( M,g\right) $ satisfies
\begin{equation*}
1\leq S\leq K\text{ \ on }B_{p}\left( 2R\right) 
\end{equation*}%
for some constant $K>1.$ Then 
\begin{equation*}
\int_{B_{p}\left( R\right) }S\leq 8\pi R+CK\omega(R)+CK\mathrm{V}(M\setminus M_{t_0}).
\end{equation*}%
for all $R>2$, where $C>0$ is some universal constant and $\omega$ is given by (\ref{omega}). 
\end{theorem}

\begin{proof}
Integrating the Bochner formula 
\begin{equation*}
\Delta \left\vert \nabla \rho \right\vert =\left( \left\vert \nabla ^{2}\rho
\right\vert ^{2}-\left\vert \nabla \left\vert \nabla \rho \right\vert
\right\vert ^{2}+\mathrm{Ric}\left( \nabla \rho ,\nabla \rho \right) \right)
\left\vert \nabla \rho \right\vert ^{-1}+\left\langle \nabla \left( \Delta
\rho \right) ,\nabla \rho \right\rangle \left\vert \nabla \rho \right\vert
^{-1}
\end{equation*}%
on $L(t_0,t)$ we obtain that 
\begin{eqnarray}
\int_{\ell(t)}\left\vert \nabla \rho \right\vert _{\nu
}-\int_{\ell(t_0) }\left\vert \nabla \rho \right\vert _{\nu }
&=&\int_{L(t_0,t)}\Phi \left\vert \nabla \rho \right\vert 
\label{z1} \\
&+&\int_{L(t_0,t)}\left\langle \nabla \left( \Delta \rho
\right) ,\nabla \rho \right\rangle \left\vert \nabla \rho \right\vert ^{-1},
\notag
\end{eqnarray}%
where $\nu =\frac{\nabla \rho }{\left\vert \nabla \rho \right\vert }$ and 
\begin{equation*}
\Phi =\left( \left\vert \nabla ^{2}\rho \right\vert ^{2}-\left\vert \nabla
\left\vert \nabla \rho \right\vert \right\vert ^{2}+\mathrm{Ric}\left(
\nabla \rho ,\nabla \rho \right) \right) \left\vert \nabla \rho \right\vert
^{-2}.
\end{equation*}
Since $\Delta \rho=|\nabla \rho|^2-1$, it follows that 
$$
\int_{L(t_0,t)}\left\langle \nabla \left( \Delta \rho
\right) ,\nabla \rho \right\rangle \left\vert \nabla \rho \right\vert ^{-1}
=2\int_{ L(t_0,t) }\left<\nabla \vert \nabla \rho\vert, \nabla \rho\right>.
$$
By the co-area formula, we have that 
\begin{equation}
\int_{L(t_0,t)}\Phi \left\vert \nabla \rho \right\vert
=\int_{t_0}^{t}\left( \int_{\ell_{0}\left( s\right) }\Phi \right)
ds+\int_{t_0}^{t}\left( \int_{\ell_{1}\left( s\right) }\Phi \right) ds.  \label{coarea}
\end{equation}%
According to Lemma \ref{Ric}, 
\begin{equation*}
\int_{t_0}^{t}\left( \int_{\ell_{0}\left( s\right) }\Phi \right) ds\geq \frac{1}{2%
}\int_{t_0}^{t}\left( \int_{\ell_{0}\left( s\right) }S\right) ds-4\pi (t-t_0).
\end{equation*}%
By Proposition \ref{rho} and Lemma \ref{Area} we have that 
\begin{eqnarray*}
\int_{t_0}^{t}\left( \int_{\ell_{0}\left( s\right) }S\right) ds &\geq
&\int_{t_0}^{t}\left( 1-\frac{10}{s}\right) \left( \int_{\ell_{0}\left( s\right)
}S\left\vert \nabla \rho \right\vert \right) ds \\
&\geq &\int_{t_0}^{t}\left( \int_{\ell_{0}\left( s\right) }S\left\vert \nabla
\rho \right\vert \right) ds-CK\omega(t)
\end{eqnarray*}
since by (\ref{omega2t}),
$$
\int_{L(t_0,t)}\frac{1}{\rho}\leq C\int_{B_p(t+C\ln t)}\frac{1}{r}\leq C\omega(t).
$$
In conclusion,
\begin{eqnarray}
\;\;\;\;\;\int_{t_0}^{t}\left( \int_{\ell_{0}\left( s\right) }\Phi \right) ds &\geq &\frac{1%
}{2}\int_{t_0}^{t}\left( \int_{\ell_{0}\left( s\right) }S\left\vert \nabla \rho
\right\vert \right) ds-4\pi (t -t_0) \label{z100}-CK\omega(t)  
\end{eqnarray}%
for any $t<2R.$
 
Lemma \ref{l1} implies that 
\begin{equation*}
\int_{t_0}^{t}\left( \int_{\ell_{1}\left( s\right) }S\left\vert \nabla \rho
\right\vert \right) ds\leq CK\omega(t)
\end{equation*}%
for all $t<2R.$  Therefore, by co-area formula, 
\begin{equation*}
  \int_{t_0}^{t}\left( \int_{\ell_{0}\left( s\right) }S\left\vert \nabla \rho
\right\vert \right) ds\geq \int_{L(t_0,t)}S\left\vert
\nabla \rho \right\vert ^{2}-CK\omega(t).
\end{equation*}
 Plugging this in (\ref{z100}) implies that 
\begin{eqnarray}\label{z2}
\int_{t_0}^{t}\left( \int_{\ell_{0}\left( s\right) }\Phi \right) ds &\geq &\frac{1%
}{2}\int_{L(t_0,t) }S\left\vert \nabla \rho \right\vert
^{2}-4\pi (t-t_0)   -CK\omega(t).  
\end{eqnarray}

Since  $\mathrm{Ric}\geq 0,$  we have that
\begin{equation*}
\int_{t_0}^{t}\left( \int_{\ell_{1}\left( s\right) }\Phi \right) ds\geq 0.
\end{equation*}
Together with (\ref{coarea}) and (\ref{z2}), we conclude that 
\begin{eqnarray*}%\label{z3}
\int_{L(t_0,t)}\Phi \left\vert \nabla \rho \right\vert 
&\geq &\frac{1}{2}\int_{L(t_0,t)}S\left\vert \nabla \rho
\right\vert ^{2}-4\pi (t-t_0)  -CK\omega(t).  
\end{eqnarray*}
Plugging this into (\ref{z1}), we see that for all $t<2R,$
\begin{eqnarray}\label{z4} 
\;\;\;\int_{\ell(t)}\left\vert \nabla \rho \right\vert _{\nu
}-\int_{\ell(t_0)}\left\vert \nabla \rho \right\vert _{\nu }
&\geq &\frac{1}{2}\int_{L(t_0,t)}S\left\vert \nabla \rho
\right\vert ^{2}-4\pi (t-t_0)  -CK\omega(t)\\
&+&2\int_{L(t_0,t)}\left\langle \nabla \left\vert \nabla
\rho \right\vert ,\nabla \rho \right\rangle.  \notag
\end{eqnarray}

Denote with 
\begin{eqnarray}
G\left( t\right)  &:=&\int_{L(t_0,t)}\left\langle \nabla
\left\vert \nabla \rho \right\vert ,\nabla \rho \right\rangle ,  \label{z6}
\\
H\left( t\right)  &:=&\frac{1}{2}\int_{L(t_0,t)}S\left\vert
\nabla \rho \right\vert ^{2}-4\pi (t-t_0)-CK\omega(t) +\int_{\ell(t_0) }\left\vert \nabla \rho \right\vert _{\nu }.  \notag
\end{eqnarray}%
By the co-area formula, the derivative of $G$ with respect to $t$ satisfies
\begin{equation*}
G^{\prime }\left( t\right) =\int_{\ell(t)}\left\langle
\nabla \left\vert \nabla \rho \right\vert ,\nabla \rho \right\rangle
\left\vert \nabla \rho \right\vert ^{-1}=\int_{\ell(t)}\left\vert \nabla \rho \right\vert _{\nu }.
\end{equation*}%
Therefore, (\ref{z4}) can be rewritten into
\begin{equation*}
G^{\prime }\left( t\right) \geq 2G\left( t\right) +H\left( t\right).
\end{equation*}%
Equivalently,
\begin{equation}
\left( e^{-2t}G\left( t\right) \right) ^{\prime }\geq e^{-2t}H\left(
t\right)   \label{z7}
\end{equation}%
for all $t_0<t<2R.$

Integrating (\ref{z7}) in $t$ from $R$ to $2R$ implies that 
\begin{eqnarray}\label{GH}
e^{-4R}G\left(2R\right) -e^{-2R}G\left( R\right) 
&\geq &\int_{R}^{2R}e^{-2t}H\left( t\right) dt. 
\end{eqnarray}%
For any $t\in [R,2R]$ we have   $$\int_{L(t_0,t)} S\vert \nabla \rho\vert^2 \geq \int_{L(t_0,R)}S\vert \nabla \rho\vert^2  \;\;\;\text{and}\;\;\; \omega (t)\leq C\omega (R).$$ 
Therefore 
$$H(t)\geq \frac{1}{2}\int_{L(t_0,R)}S\vert \nabla \rho\vert^2  -CK\omega(R)-\beta -4\pi (t-t_0),$$ for $\beta$ defined in (\ref{Ct0}). By (\ref{Ct0 bound}), it follows that 
\begin{eqnarray*}
\int_{R}^{2R}e^{-2t}H\left( t\right) dt&\geq& \frac{1}{4}\left(e^{-2R}-e^{-4R}\right) \int_{L(t_0,R)}S\vert \nabla \rho\vert^2\\
&-&(2\pi R- R_0)e^{-2R}-C(K\omega(R)+\beta)e^{-2R}.
\end{eqnarray*}

Consequently, (\ref{GH}) implies that 
 \begin{eqnarray}\label{GH10}
\;\;\;\; \frac{1}{4}\left(e^{-2R}-e^{-4R}\right) \int_{L(t_0,R)}S\vert \nabla \rho\vert^2&\leq& (2\pi R-R_0)e^{-2R}\\
&+&C(K\omega(R)+\beta )e^{-2R}\notag \\
&+&
e^{-4R}G\left(2R\right) -e^{-2R}G\left( R\right).\notag
\end{eqnarray}

On the other hand, integrating by parts in (\ref{z6}) gives
 
\begin{equation}\label{GH2}
G\left( t\right)  =-\int_{L(t_0,t)}\left\vert \nabla
\rho \right\vert \Delta \rho +\int_{\ell(t)}\left\vert
\nabla \rho \right\vert ^{2}-\int_{\ell(t_0)}\left\vert
\nabla \rho \right\vert ^{2}. 
\end{equation}
By Proposition \ref{rho}, the first term is estimated as 
$$
\left\vert\int_{L(t_0,t)}\left\vert \nabla
\rho \right\vert \Delta \rho\right\vert \leq C\int_{L(t_0,t)}\left\vert \nabla
\rho \right\vert  \leq Ct^3,
$$
where the volume comparison theorem is used. Thus,
$$
G(2R)\leq \int_{\ell(2R)}\left\vert
\nabla \rho \right\vert ^{2}+CR^3.
$$
Applying Proposition \ref{rho} again, we have
 
\begin{eqnarray*}
\int_{\ell(2R)}\left\vert \nabla \rho \right\vert ^{2}&\leq& C\int_{\ell(2R)}\left\vert
\nabla \rho \right\vert \\
&=&C\int_{L(t_0,2R)}\Delta \rho +C\int_{\ell(t_0)}\left\vert
\nabla \rho \right\vert \\
&\leq &CR^3+C\beta.
\end{eqnarray*}
This proves that 
\begin{equation}\label{GH3}
G(2R)\leq CR^3+C\beta.
\end{equation}
 We now obtain a lower bound for $G$. We have 
\begin{eqnarray*}
G\left( R\right) &=& \int_{L(t_0,t)}\left\vert \nabla
\rho \right\vert \left(1-|\nabla \rho|^2\right) +\int_{\ell(t)}\left\vert
\nabla \rho \right\vert ^{2}-\int_{\ell(t_0)}\left\vert
\nabla \rho \right\vert ^{2}.
\\ &\geq &\int_{L(t_0,R)}\left\vert \nabla
\rho \right\vert \left( 1-\left\vert \nabla \rho \right\vert ^{2}\right) -\beta\\
&\geq&-C\int_{L(t_0,R)}\frac{1}{\rho}-\beta\\
&\geq &-C\omega(R)-\beta. 
\end{eqnarray*}
We conclude that 
$$
e^{-4R}G\left(2R\right) -e^{-2R}G\left( R\right) \leq C\left(\omega(R)+\beta\right)e^{-2R}+CR^3 e^{-4R}.
$$
Plugging this into (\ref{GH10})  implies that  
\begin{equation}
\int_{L(t_0,R)}S\left\vert \nabla \rho
\right\vert ^{2}\leq 8\pi R+CK\omega(R)+C\beta.  \label{z9}
\end{equation}
Since $S\leq K$ on $B_p\left(2R\right) $ and $\left\vert \nabla \rho
\right\vert ^{2}\leq 1+\frac{C}{\rho}$, it follows that 
\begin{eqnarray*}
0 &\leq &\left( S-K\right) \left( \left\vert \nabla \rho \right\vert ^{2}-1-%
\frac{C}{\rho}\right)  \\
&\leq &S\left\vert \nabla \rho \right\vert ^{2}-S+K-K\left\vert \nabla \rho
\right\vert ^{2}+\frac{CK}{\rho} \\
&=&S\left\vert \nabla \rho \right\vert ^{2}-S-K\Delta \rho +\frac{CK}{\rho}.
\end{eqnarray*}%
Hence,
 
\begin{equation*}
\int_{L(t_0,R)}S\left\vert \nabla \rho \right\vert
^{2}\geq \int_{ L(t_0,R) }S+K\int_{L(t_0,R)}\Delta \rho -CK\int_{L(t_0,R) }\frac{1}{\rho }.
\end{equation*}%
Note that%
\begin{equation*}
\int_{ L(t_0,R) }\Delta \rho =\int_{ \ell(R)}\left\vert \nabla \rho \right\vert -
\int_{ \ell(t_0)}\left\vert \nabla \rho \right\vert \geq -\beta
\end{equation*}
and that 
$$
\int_{L(t_0,R) }\frac{1}{\rho }\leq C\omega(R).
$$
Therefore, 
\begin{equation}
\int_{ L(t_0,R)  }S\left\vert \nabla \rho \right\vert
^{2}\geq \int_{ L(t_0,R)  }S-CK\omega(R)-CK\beta.  \label{z10}
\end{equation}%
From (\ref{z9}) and (\ref{z10}) it follows that 
\begin{equation*}
\int_{L(t_0,R)}S\leq 8\pi R+CK\omega(R)+CK\beta.
\end{equation*} 
However, $B_p (R)\setminus L(t_0,R)=M\setminus M_{t_0}$, therefore 
\begin{equation*}
\int_{B_p (R)}S\leq 8\pi R+CK\omega(R)+CK\mathrm{V}(M\setminus M_{t_0}).
\end{equation*} 
In view of Proposition \ref{rho}, the theorem follows.
\end{proof}

\end{document}